\documentclass{amsart}
\usepackage[dvipsnames]{xcolor}
\usepackage{amsmath,amssymb,tikz,tikz-cd}
\usepackage{hyperref}
\usepackage[margin=1.6in]{geometry}
\usepackage[nameinlink,capitalize]{cleveref}
\newcommand\myshade{85}
\colorlet{mylinkcolor}{violet}
\colorlet{mycitecolor}{YellowOrange}
\colorlet{myurlcolor}{Aquamarine}

\hypersetup{
  linkcolor  = mylinkcolor!\myshade!black,
  citecolor  = mycitecolor!\myshade!black,
  urlcolor   = myurlcolor!\myshade!black,
  colorlinks = true,
}

\newtheorem{theorem}{Theorem}[section]
\newtheorem*{theorem*}{Theorem}
\newtheorem{proposition}[theorem]{Proposition}
\newtheorem{lemma}[theorem]{Lemma}
\newtheorem{corollary}[theorem]{Corollary}

\theoremstyle{definition}
\newtheorem{definition}[theorem]{Definition}
\newtheorem{example}[theorem]{Example}
\newtheorem{remark}[theorem]{Remark}

\newcommand{\EA}{\operatorname{EA}}
\newcommand{\EP}{\operatorname{EP}}
\newcommand{\IA}{\operatorname{IA}}
\newcommand{\IP}{\operatorname{IP}}
\newcommand{\supp}{\operatorname{supp}}
\newcommand{\0}{\emptyset}
\def\NBC{{\textsc{NBC}}}
\def\aNBC{\Delta^{\textup{nbc}}}

\title{The augmented External Activity Complex of a Matroid} 
\author{Andrew Berget}
\author{Dania Morales}

\begin{document}
\begin{abstract}
    For a matroid, we define a new simplicial complex whose facets are indexed by its independent sets. This complex contains the  external activity complex as a subcomplex. We call our complex the augmented external activity complex since its definition is motivated by the recently defined augmented tautological classes of matroids. We prove that our complex is shellable and show that our shelling satisfies the stronger property of being an $H$-shelling. This explicates our result that the $h$-vector of our complex is the $f$-vector of the independence complex. We also define an augmented no broken circuit complex, which contains the usual no broken circuit complex as a subcomplex. We prove its shellability and show that our shelling is also an $H$-shelling. The $h$-vector of this complex is the $f$-vector of the no broken circuit complex.
\end{abstract}
\maketitle
\section{Introduction}
While any simplicial complex can be assembled one face at a time in an apparently haphazard way, \textit{shellable} simplical complexes are those that can be assembled with some measure of control on the topology of this process. One requires that the facets of the complex can be added in such an order that when adding a new facet, it meets the previously built complex along a pure subcomplex. Such an ordering of the facets is called a shelling of the complex, and among other important consequences, shellability implies that a simplicial complex is Cohen-Macaulay and homotopy equivalent to a wedge of spheres.

Matroids give rise to supernumerary examples of shellable simplicial complexes. Two of the most important and earliest results in this area are the shellability of the independence complex and the order complex of the lattice  of flats of a matroid, due to Provan and Billera, and Garsia, respectively. Provan also showed that the no broken circuit complex of $M$ is shellable. For all of these results in one place, see \cite{bjorner}. The Bergman complex of a matroid, used in the resolution of the Rota-Heron-Welsh conjecture \cite{AHK}, is homeomorphic to the order complex of its lattice of proper flats \cite{ardilaklivans} and is, thus,  shellable. The augmented Bergman complex of a matroid, used in the resolution of the Dowling-Wilson conjecture \cite{braden}, interpolates between the independence complex and the lattice of flats, and was recently shown to be shellable \cite{reiner}. An interesting phenomenon appears in each of these examples: it is possible to describe not just one, but many different shellings of these complexes. As a simplest example, any linear order of the ground set of a matroid induces a lexicographic ordering on its bases, and this is a shelling. This property can be used to characterize matroids among all simplicial complexes.

In \cite{ardilaboocher} a curiously defined complex called the \textit{external activity complex} was discovered. This complex arises naturally in the study of what is now called the \textit{Schubert variety of a linear space}, which was used in the proof of the Dowling-Wilson conjecture of a realizable matroid \cite{huhwang}. This variety is almost always singular. A consequence of \cite{ardilaboocher} is the Cohen-Macaulay property of the external activity complex, which means that it has, in a sense, mild singularities. In \cite{ACS} the topology of the external activity complex was studied in greater depth, and it was shown to be shellable.

The facets of the external activity complex are in bijection with the bases of the matroid, however they record finer information than just the bases: they record each basis $B$ along with its \textit{external activity} $\EA(B)$, which is to say the elements $e$ not in the basis $B$ that are the maximum of the unique circuit of $B \cup \{e\}$. External activity is a concept dating back to Whitney's description of the chromatic polynomial of a graph in terms of no broken circuit sets  \cite{whitney} and made formal in Crapo's definition of the Tutte polynomial \cite{crapo}.  There is a notion of internal activity $\IA(B)$ of a basis,  which is obtained from external activity by matroid duality. In a systematic study of orderings of bases of matroids that are known to give shellings, Las Vergnas \cite{LV} introduced partial orders on the bases of a matroid coming from internal and external activity. The main result in \cite{ACS} shows that linear extensions of Las Vergnas's external/internal partial order induce shelling orders on both the external activity complex and the independence complex, which is a subcomplex of the external activity complex.

In the current work, we describe a naturally occurring simplicial complex whose facets are indexed by the independent sets of a matroid. The facet of an independent set records, essentially, its internal and external activity. Our main results are summarized below. We use the notation $x_S = \prod_{i \in S}x_i$ to describe monomials, and a simplicial complex is described by a listing a collection of square free monomials corresponding to its facets. 
\begin{theorem*}
Let $M$ be a matroid of rank $r$ on the set $E$.  Define a simplicial complex $\Delta_M$, the augmented external activity complex of $M$, with ground set $\{x_e, y_e,z_e : e \in E\}$ and a facet,
\[ x_{I \cup \EP(I)} y_Y z_{I \cup \EA(I)}\]
for each independent set $I$ of $M$, where
$I = B\setminus Y$, with $B$ a basis and $Y \subseteq \IA(B)$. Here $\EA, \EP, \IA$ are external activity and passivity, and internal activity, respectively. Then,
\begin{enumerate}
    \item $\Delta_M$ contains the external activity complex of \cite{ACS} as a subcomplex.
    \item Any linear extension of the Las Vergnas's external/internal order on independent sets of $M$ gives a shelling of $\Delta_M$. In particular, $\Delta_M$ is Cohen-Macaulay.
    \item The $h$-vector of $\Delta_M$ is the $f$-vector of the independence complex of $M$.
    \item $\Delta_M$ contains a shellable subcomplex $\aNBC_M$, itself containing the no broken circuit complex of $M$, $\NBC(M)$. The $h$-vector of $\aNBC_M$ is equal to the $f$-vector of $\NBC(M)$.
\end{enumerate}
\end{theorem*}
While our results might at first appear to be a modest extension of the results of \cite{ACS}, it is worth emphasizing a main difficulty in their genesis: simply coming up with a definition.
\subsection{Geometric Motivation}\label{ssec:motivation}
We include here the geometric motivation for our definition and results. This material is not needed to understand the combinatorics in the body of our paper, and can be skipped if desired.  In spite of this, we feel it is important to emphasize that our complex $\Delta_M$ was not divined out of the ether or was an obvious generalization of \cite{ACS}, but had a concrete and interesting geometric origin. We will not prove the results stated here, reserving these for a future work. 

The Schubert variety of a linear space $L \subset \mathbf{C}^n$ is obtained by embedding $\mathbf{C}^n \subset (\mathbf{P}^1)^n$ and taking the closure of $L$ in $(\mathbf{P}^1)^n$. The resulting variety $Y_L$ is multiplicity free in the sense that its multidegree, which is a polynomial in $t_1,\dots,t_n$, is the basis generating function of the matroid $M$ of $L$. By results of Brion \cite{brion} this implies $Y_L$ is normal, has rational singularities and is Cohen-Macaulay. It also implies that $Y_L$ flatly deforms to a reduced union of Schubert varieties in $(\mathbf{P}^1)^n$. 

Let $\hat{Y}_L \subset (\mathbf{C}^2)^n$ be the multiaffine cone of $Y_L$. Viewing $(\mathbf{C}^2)^n$ as $\mathbf{C}^n \times \mathbf{C}^n$ affords a different description of $\hat Y_L$, coming from the tautological bundles of linear spaces in \cite{best}: $\hat Y_L$ is obtained from the sum of tautological bundles $\mathcal{S}_L \oplus \mathcal{O}(-\beta)$ on the permutohedral variety $\underline{X_n}$ as 
\[
\begin{tikzcd}
\mathcal{S}_L \oplus \mathcal{O}(-\beta) \arrow{r} \arrow{d} &  \mathbf{C}^n \times \mathbf{C}^n \times \underline{X_n} \arrow{d}\\
\hat Y_L  \arrow{r} & \mathbf{C}^n \times \mathbf{C}^n
\end{tikzcd}
\]
Here the horizontal arrows are inclusions and the vertical arrows are projections. That is, we may define $\hat Y_L$ to be the image of the bundle under the projection mapping. The computation of the $\mathbf{Z}^2$-graded multidegree in \cite{ardilaboocher}, for example,  now becomes a computation involving the Segre classes of $\mathcal{S}_L$ and $\mathcal{O}(-\beta)$, which is done in \cite{best} and yields (essentially) $T_M(q,1)$. From this point of view, it is natural to extend the construction of the Schubert variety of a linear space in any number of ways by varying the bundle. The main ingredient is a subbundle of a trivial bundle over a smooth projective variety.

Our complex $\Delta_M$ comes from changing the permutohedral variety $\underline{X_n}$ to the stellahedral variety $X_n$, and modifying the bundles involved to be the augmented tautological bundles defined in \cite{EHL}. Specifically, we use the bundle $\mathcal{E}_L:=\mathcal{Q}_L^\vee \oplus \pi^*_E(\mathcal{O}_{\mathbf{P}^E}(1))^\vee$, which is a subbundle of a trivial bundle with fiber $\mathbf{C}^n \times \mathbf{C}^n \times \mathbf{C}^{n+1}$. Applying \cite[Theorem~1.11]{EHL}, we see that the subvariety $\hat E_L$ of $\mathbf{C}^n \times \mathbf{C}^n \times \mathbf{C}^{n+1}$ produced by $\mathcal{E}_L$ will have degree equal to the number of independent subsets of the matroid $M$ of $L$. We may view $\hat E_L$ as the multi-affine cone over a subvariety $E_L \subset (\mathbf{P}^1)^n \times \mathbf{P}^n$. It is immediate from the definition and the main theorem of \cite{binglin} that $E_L$ is multiplicity free, and that its $\mathbf{Z}^n$  multidegree is the generating function for the independent sets of $M$. We can thus apply the same results of Brion above, and conclude that $E_L$ has rational singularities, is Cohen-Macaulay and flatly deforms to a reduced union of Schubert varieties.

Armed with these results, one hopes that $\hat E_L$ has a Gr\"obner deformation (which is flat) to a Cohen-Macaulay simplicial complex that only depends on the matroid of $M$. This is indeed what happens and our complex $\Delta_M$ is the result of a careful analysis of the initial ideal under an appropriate term order. Implicit in this is  the result that the initial ideal of $\hat E_L$ only depends on the matroid $M$, and so we are able to give a definition that makes no reference to the geometric motivators of linear spaces, tautological bundles or varieties. As a consequence, our main theorem is entirely combinatorial. It is a natural strengthening of the statement that $E_L$ is Cohen-Macaulay.

The proof of our main theorem builds upon the techniques used in \cite{ACS}, which involves an intricate analysis of Las Vergnas's external/internal order on the bases of a matroid. Our complex $\Delta_M$ contains the external activity complex of $M$ as a subcomplex and the argument we employ reduces to the argument in \cite{ACS} when restricted to this subcomplex. However, we need to extend Las Vergnas's result to define an external/internal order on all independent sets, and extra care must be taken here as new obstacles are encountered. 

\subsection{Organization} The structure of our paper is as follows: In \cref{sec:act,sec:shell} we review known material on matroid activities and shellings. In \cref{sec:acs} we review the construction and results of \cite{ACS}, highlighting the results we will need later. In \cref{sec:augea} we define the augmented external activity complex and prove our main result on its shellability. We include a fully worked  example to aid the reader in following the delicacies of our proof. In \cref{sec:restriction} we describe the restriction sets of our shellings, as well as a related shelling which produces a two-variable shelling polynomial. We use our knowledge of the restriction sets to describe the $h$-vector of our complex and to prove that our shellings are a special kind of shelling called an $H$-shelling, studied previously in \cite{ER}. Finally, in \cref{sec:nbc} we define an augmented analogue of the no broken circuit complex of a matroid, proving results parallel to those in \cref{sec:augea,sec:restriction}.

\subsection*{Acknowledgements} The authors would like to thank the Kennerud Visiting Math Scholars Fund at Western Washington University for helping facilitate the nascent stages of this project. AB would like to thank Colin Crowley for useful discussions.

\section{Matroid Activities}\label{sec:act}
A matroid $M$ is a pair $(E,\mathcal{I})$ where $E$ is a finite set and $\mathcal{I} = \mathcal{I}(M)$ is a simplicial complex on $E$ satisfying the following axiom:
\begin{center}
    \textit{if $I,J \in \mathcal{I}$ and $|I|<|J|$, then there is $e \in J \setminus I$ with $I \cup \{e\} \in \mathcal{I}$}.
\end{center}
The sets in $\mathcal{I}$ are the independent sets of $M$ and the facets of $\mathcal{I}$ are the bases of $M$. We denote the collection of bases of $M$ by $\mathcal{B} = \mathcal{B}(M)$. We refer the reader to the book of Oxley \cite{oxley} for basic elements of matroid theory. We will recall further, less basic, elements of the theory below.

\subsection{Activities}\label{ssec:activitites} We recall the definitions for matroid external and internal activities, following Las Vergnas \cite{LV}. We abuse notation slightly in this section and all that follow, by writing the union (and set difference) of a set $B$ with a singleton set $\{ e \}$ by suppressing the set braces for the singleton set. That is, we denote $B\cup \{ e \} := B \cup e$ and similarly we denote $B\setminus \{ e \}:= B\setminus e$.


Let $M$ be a matroid on $E$ and fix once and for all an arbitrary linear order $<$ on $E$.
\begin{remark}
    There are three orders that appear in our paper that may be difficult to keep track of on a first reading. The first is the one just introduced, $<$, which is a total order on $E$, the ground set of our matroid. The second is $\leq_{ext/int}$, which will first be a partial order on the bases and then extended to a partial order on the independent sets of $M$. Finally, there will be a linear extension $\prec$ of $\leq_{ext/int}$. Which order is being used will make it clear what types of objects are being compared. 
\end{remark}
\begin{definition}
Let $S$ be a subset of $E$. An element $e\in E \setminus S$ is called \emph{externally active} with respect to $S$ (and $M$) if there exists a circuit $\gamma$ of $M$ contained in $S \cup e$ so that $e$ is the maximum element in $\gamma$. Otherwise, $e$ is called \emph{externally passive} with respect to $S$ (and $M$). 
Denote the set of elements that are externally active (externally passive) with respect to $S$ by $\EA(S)$ (resp., $\EP(S)$). 
\end{definition}
Typically we will not say ``externally active for $S$ and $M$'', but only ``externally active for $S$''. However, we need the enhanced emphasis on $M$ to make the next definition.
\begin{definition}
 An element $i\in S$ is called \emph{internally active} with respect to $S$ (and $M$) if $i$ is externally active with respect to $E\setminus S$ and $M^\perp$, where $M^\perp$ is the dual matroid with the same order on the ground set. Otherwise $i \in S$ is called \emph{internally passive} with respect to $S$ (and $M$).
Denote the set of elements that are internally active (internally passive) with respect to $S$ by $\IA(S)$ (resp., $\IP(S)$).
\end{definition} 

For a basis $B$ of $M$, note that $e \notin B$ is externally active with respect to $B$ if and only if there is no element $e' >e$ with $e' \in B$ and $B\setminus e' \cup e$ a basis of $M$. Dually, $e \in B$ is internally active with respect to $B$ if and only if there is no element $e' >e$ with $e' \notin B$ and $B\setminus e \cup e'$ a basis of $M$. 

The following result of Crapo was used in his definition of the Tutte polynomial of a matroid.
\begin{proposition}\label{prop:crapo}\cite{crapo} Let $M$ be a matroid on the ground set $E$ and let $<$ be a linear order on $E$.
\begin{enumerate} 
\item Every subset $S$ of $E$ can be uniquely written in the form $S = B\setminus Y \cup X$ for some basis
$B$, with $X \subset \EA(B)$ and $Y \subset \IA(B)$. Equivalently, the intervals
$[B\setminus \IA(B),B\cup \EA(B)]$ form a partition of the poset of subsets of $E$ ordered by inclusion.
\item Every independent set $I$ of $E$ can be uniquely written in the form $I = B\setminus Y$ for some basis $B$ and some subset $Y \subset \IA(B)$. Equivalently, the intervals $[B\setminus \IA(B),B]$ form a partition of the independence complex, $\mathcal{I}(M)$.
\end{enumerate}
\end{proposition}
The second item here will be particularly important for us.
\subsection{Active orders for matroid bases}\label{ssec:activeorders}
In his study of matroid activities Las Vergnas introduced three partial orders on the collection of bases of $M$. We summarize a few of his results here, taking them as definitions.
\begin{definition}\cite[Proposition 3.1]{LV} We define the \emph{external order} on $\mathcal{B}(M)$ by the following equivalent conditions
\begin{enumerate}
\item $A\leq_{ext}B$
\item $A\subset B \cup \EA(B)$
\item $A\cup \EA(A) \subset B \cup \EA(B)$
\end{enumerate}
\end{definition}
Dually, we have the following.
\begin{definition}\cite[Proposition 5.2]{LV}\label{LVint} We define the \emph{internal order} on $\mathcal{B}(M)$ by the following equivalent conditions
\begin{enumerate}
\item $A\leq_{int}B$
\item $A\setminus \IA(A)\subset B$
\item $A\setminus \IA(A) \subset B \setminus \IA(B)$
\end{enumerate}
\end{definition}
The external/internal order on bases is the weakest order that simultaneously
extends the external and internal orders.
\begin{definition}
\cite[Proposition 6.3]{LV}\label{LVextint} We define the \emph{external/internal order} on $\mathcal{B}(M)$ by the following equivalent conditions
\begin{enumerate}
\item $A\leq_{ext/int}B$
\item $\IP(A)\cap \EP(B) = \emptyset$
\item $A\setminus \IA(A)\cup \EA(A) \subset B \setminus \IA(B) \cup \EA(B)$
\item $\IP(A) \cup \EA(A) \subset \IP(B) \cup \EA(B)$
\end{enumerate}    
\end{definition}
\begin{example}\label{ex:running1} Consider the following running example. Let $M$ be the matroid on $E = \{1,2,3,4,5\}$ (with the natural order) realized by the affine point configuration below.
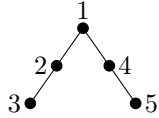
\[
\begin{tikzpicture}
\node at (0,1) [above] {$1$};
\node at (-.35,.5) [left] {$2$};
\node at (-.7,0) [left] {$3$};
\node at (.35,.5) [right] {$4$};
\node at (.7,0) [right] {$5$};
\draw[fill=black] (0,1) circle (2pt);
\draw[fill=black] (-.35,.5) circle (2pt);
\draw[fill=black] (-.7,0) circle (2pt);
\draw[fill=black] (.35,.5) circle (2pt);
\draw[fill=black] (.7,0) circle (2pt);
\draw (-.7,0)-- (-.35,.5)--(0,1)--(.35,.5)--(.7,0);
\end{tikzpicture}
\]
We summarize here the bases along with their activities.
\[
\begin{array}{c|cccc}
B& \EA(B) & \EP(B) & \IA(B) & \IP(B) \\
\hline
345 & \0 & 12 & 345 & \0  \\
135 & \0 & 24 & 35 & 1  \\
245 & \0 & 13 & 45 & 2  \\
235 & \0 & 14 & 5 & 23  \\
125 & 3 & 4 & 5 & 12  \\
134 & 5 & 2 & 3 & 14 \\
234 & 5 & 1 & \0 & 234 \\
124 & 35 & \0 & \0 & 124 \\
\end{array}
\]
The Hasse diagrams of the external and internal orders on $\mathcal{B}(M)$ are shown at left and right below.
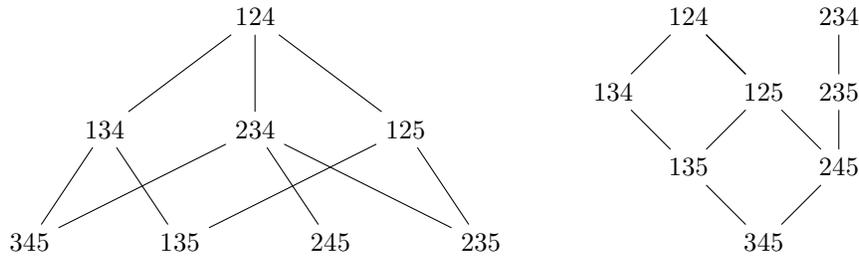
\[\begin{tikzpicture}[scale=1]
\node (345) at (-3,0) {$345$};
\node (135) at (-1,0) {$135$};
\node (235) at (3,0) {$235$};
\node (245) at (1,0) {$245$};
\node (2345) at (0,1.5) {$234$};
\node (1345) at (-2,1.5) {$134$};
\node (1235) at (2,1.5) {$125$};
\node (12345) at (0,3) {$124$};
\draw (345)--(1345)--(12345);
\draw (135)--(1345);
\draw (135)--(1235);
\draw (235)--(1235)--(12345);
\draw (235)--(2345);
\draw (345)--(2345);
\draw (245)--(2345)--(12345);
\end{tikzpicture}
\hspace{1cm}
\begin{tikzpicture}[scale=1]
\node (0) at (0,0) {$345$};
\node (1) at (-1,1) {$135$};
\node (2) at (1,1) {$245$};
\node (23) at (1,2) {$235$};
\node (12) at (0,2) {$125$};
\node (14) at (-2,2) {$134$};
\node (234) at (1,3) {$234$};
\node (124) at (-1,3) {$124$};
\draw (0)--(1)--(12)--(124);
\draw (1)--(14)--(124);
\draw (0)--(2)--(23)--(234);
\draw (2)--(12)--(124);
\end{tikzpicture}
\]
The external/internal  order on $\mathcal{B}(M)$ refines both of these, and its Hasse diagram is shown below.
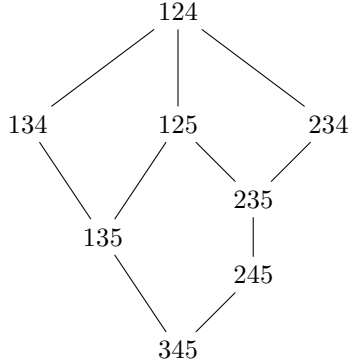
\[
\begin{tikzpicture}[scale=1]
\node (0) at (0,0) {$345$};
\node (1) at (-1,1.5) {$135$};
\node (2) at (1,1) {$245$};
\node (23) at (1,2) {$235$};
\node (123) at (0,3) {$125$};
\node (145) at (-2,3) {$134$};
\node (2345) at (2,3) {$234$};
\node (12345) at (0,4.5) {$124$};
\draw (0)--(1)--(123)--(12345);
\draw (1)--(145)--(12345);
\draw (0)--(2)--(23)--(123);
\draw (23)--(2345)--(12345);
\end{tikzpicture}
\]
\end{example}

\subsection{Active orders for independent sets}
We wish to extend Las Vergnas's external/internal order to the collection of independent sets of a matroid. To do this, we associate each independent set to a basis as in \cref{prop:crapo} and consider external activity and internal passivity.  

\begin{definition}
We say that independent sets $I$ and $J$ of $M$ are \emph{internally related} if there exists a basis $B$ of $M$ with $I=B\setminus Y_I$ and $J=B\setminus Y_J$ where $Y_I,Y_J\subset \IA(B)$. When no confusion will arise, we will sometimes simply say $I$ and $J$ are \emph{related}.
\end{definition}
Note that by \cref{prop:crapo}(2) every independent set is related to a unique basis and every basis is related to itself. The following result is immediate from the definition.
\begin{proposition}\label{prop:delete}
Suppose that $I$ is related to a basis $B$ and $Y\subset \IA(B)$. Then $I\setminus Y$ is related to $B$.

\end{proposition}
The following result describes how the activities of independent sets are determined by their related bases. We will use it frequently.
\begin{proposition}[{\cite[Proposition~2.4(v)]{LV}}]\label{prop:lvprop}
Let $I$ be an independent set of $M$ related to the basis $B$. Then $\EA(I) = \EA(B)$ and $\IP(I) = \IP(B)$. Moreover, \[I\setminus \IA(I) \cup \EA(I)=B\setminus \IA(B) \cup \EA(B).\]
\end{proposition}

The following definition extends the external/internal order on bases to all independent sets.
\begin{definition}\label{def:extint}
 Let $I$ and $J$ be independent sets of $M$. We say that $I\leq_{ext/int}J$ if and only if $I$ and $J$ are not internally related and $I\setminus \IA(I) \cup \EA(I) \subset J\setminus \IA(J) \cup \EA(J)$ or $I$ and $J$ are internally related and $I \subset J$.
\end{definition}
One checks that this is indeed a partial order. It is immediate that two bases $B$ and $C$ satisfy $B \leq_{ext/int} C$ in the previous sense if and only if they do in the current sense. Further, if an independent set $I$ is internally related to a basis $C$ then $I \leq_{ext/int} C$. More generally, we have the following.
\begin{proposition}
Let $I$ and $K$ be independent sets related to bases $A$ and $C$, respectively with $A\neq C$. We have that $I\leq_{ext/int}K$ if and only if $A\leq_{ext/int}C$.
\end{proposition}

\begin{example}\label{ex:running2}
We continue \cref{ex:running1}. The Hasse diagram of the external/internal order on the independent sets of $M$ is shown below.
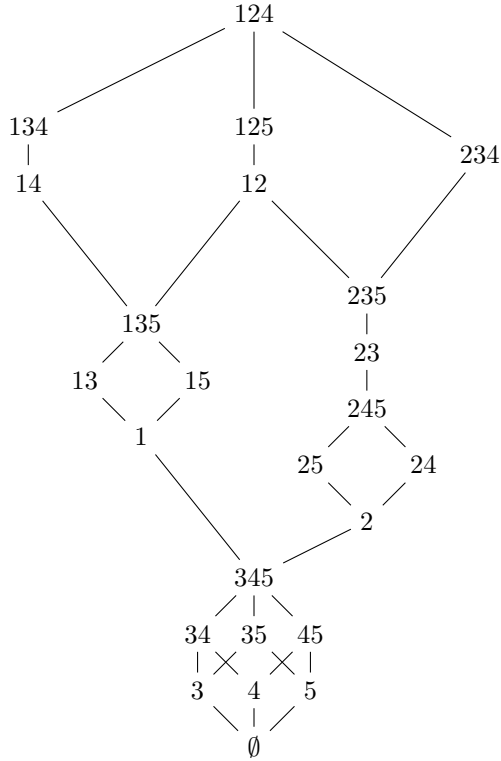
\[
\begin{tikzpicture}[scale=1.5]
\node (345) at (0,1) {$345$};
\node (45) at (.5,.5) {$45$};
\node (35) at (0,.5) {$35$};
\node (34) at (-.5,.5) {$34$};
\node (5) at (.5,0) {$5$};
\node (4) at (0,0) {$4$};
\node (3) at (-.5,0) {$3$};
\node (0) at (0,-.5) {$\emptyset$};
\draw (0)--(3)--(34)--(345);
\draw (0)--(5)--(45)--(345);
\draw (3)--(35);
\draw (4)--(34);
\draw (4)--(45);
\draw (5)--(35);
\draw (0)--(4);
\draw (35)--(345);
\node (135) at (-1,3.25) {$135$};
\node (13) at (-1.5,2.75) {$13$};
\node (15) at (-.5,2.75) {$15$};
\node (1) at (-1,2.25) {$1$};
\draw (1)--(13)--(135);
\draw (1)--(15)--(135);
\node (2) at (1,1.5) {$2$};
\node (25) at (.5,2) {$25$};
\node (24) at (1.5,2) {$24$};
\node (245) at (1,2.5) {$245$};
\draw (2)--(25)--(245);
\draw (2)--(24)--(245);
\node (235) at (1,3.5) {$235$};
\node (23) at (1,3) {$23$};
\draw (23)--(235);
\node (134) at (-2,5) {$134$};
\node (14) at (-2,4.5) {$14$};
\draw (14)--(134);
\node (125) at (0,5) {$125$};
\node (12) at (0,4.5) {$12$};
\draw (12)--(125);
\node (234) at (2,4.75) {$234$};
\node (124) at (0,6) {$124$};
\draw (345)--(1);
\draw (345)--(2);
\draw (245)--(23);
\draw (135)--(14);
\draw (135)--(12);
\draw (235)--(12);
\draw (235)--(234);
\draw (134)--(124);
\draw (125)--(124);
\draw (234)--(124);
\end{tikzpicture}
\]
Note that in this example, we replace each basis $C$ in \cref{ex:running1} with a copy of a boolean lattice of sets of the form $\{S \cup \IP(C) : S \subset \IA(C)\}$.
\end{example}
The phenomenon of this example persists in general.
\begin{proposition}\label{prop:bool}
  Let $B$ and $C$ be bases of $M$ with $C$ covering $B$ in the external/internal order on bases. Then, in the external/internal order on independent sets the interval $(B,C]$ can be described as the boolean lattice, 
  \[(B,C] = \{ S\cup \IP(C): S \subset \IA(C) \}.
  \]
\end{proposition} 
\begin{proof}
    Assume that $I$ is such that $B <_{ext/int} I \leq_{ext/int} C$. We claim that $I$ is internally related to $C$. If not then $I$ is internally related to a basis $C' \neq C$. Since $I \setminus \IA(I) \cup \EA(I) = C'\setminus \IA(C') \cup \EA(C')$, it follows that $B<_{ext/int}C'<_{ext/int}C$, which is a contradiction. 

    Since $I$ is related to $C$, it is obtained by deleting some subset of internally active elements from $C$. It follows that $I = S \cup \IP(C)$ where $S \subset \IA(C)$. The restriction of $\leq_{ext/int}$ to subsets of this form gives a boolean lattice by definition.    
\end{proof}
\begin{remark}\label{rem:auto}
The proposition displays a curious non-trivial automorphism on the collection of independent sets of matroid: Every closed interval $(B,C]$ can be ``flipped upside down''. That is, the map $S \cup \IP(C) \mapsto (\IA(C)\setminus S) \cup \IP(C)$ is an involutive bijection $\mathcal{I}(M) \to \mathcal{I}(M)$. It follows from this that
\[
\sum_{I \in \mathcal{I}(M)} q^{|I|}
=
\sum_{I \in \mathcal{I}(M)} q^{|\IA(C_I)-Y_I|+|\IP(C_I)|}
\]
where we write $I = C_I \setminus Y_I$ where $C_I$ is a basis and $Y_I \subset \IA(C_I)$. 
\end{remark}

Las Vergnas shows that on $\mathcal{B}(M)$, $\leq_{ext/int}$ defines a lattice \cite[Theorem~6.4]{LV}. Denote the meet and join by $\wedge^{\mathcal{B}}$ and $\vee^{\mathcal{B}}$. We offer the following extension of this result.
\begin{proposition}
 The order relation $\leq_{ext/int}$ defines a lattice on $\mathcal{I}(M)$, with meet $\wedge$ and join $\vee$ given by 
 \begin{enumerate}
     \item $I \wedge J = I \cap J$, $I \vee J = I \cup J$ whenever $I$ and $J$ are internally related;
     \item If $I$ and $J$ are related to different  bases $A$ and $B$, then $I \wedge J = A \wedge^{\mathcal{B}} B $ and $I \vee J = \IP(A \vee^{\mathcal{B}} B)$.
 \end{enumerate}
\end{proposition}
\begin{proof}
    If $I$ and $J$ are internally related to a basis $C$ this is immediate from \cref{prop:bool}. If $I$ and $J$ are not related, the uniqueness of greatest lower bounds and least upper bounds is determined by the uniqueness of taking the meet and join of the related bases.
\end{proof}

\section{Shellings}\label{sec:shell}
Here we introduce some background on shellings of simplicial complexes.
\subsection{Shellable complexes}
\label{ssec:shellability} We offer the first two sections of \cite{bjorner} as a particularly relevant reference on shellability of simplicial complexes, $h$-polynomials, etc. as it relates to matroids and related complexes.

Recall that a simplicial complex is pure if all its facets have the same cardinality. The dimension of a face of a simplicial complex is one less than its cardinality.
\begin{definition}\label{def:shellable} Let $\Delta$ be a pure simplicial complex with dimension $d$. A \emph{shelling order} is an ordering of the facets
$F_1,F_2,\ldots,F_s$ such that for every $i<k$ there exists $j<k$ and $e \in F_k$ with $F_i \cap F_k \subseteq F_j \cap F_k = F_k \setminus e$.
If a shelling order exists, then we call $\Delta$ \emph{shellable}.
\end{definition}
It is equivalent to demand that, for all $2 \leq j \leq s$, the complex generated by $F_1, \dots, F_{j-1}$ meets the complex generated by $F_j$ in a pure subcomplex of dimension $d-1$. 

Given a shelling order and a facet $F_j$, there is a subset $R_j$ such that for every $A \subseteq F_j$, we have that $R_j \subseteq A$ if and only of $A \not\subseteq F_i$ for every $i < j$. That is, when we add the facet $F_j$ to the  subcomplex generated by $F_1,\dots,F_{j-1}$, the new faces that we introduce are precisely those in the interval $[R_j, F_j]$. The set $R_j$ is called the \emph{restriction set} of $F_j$ in the shelling order.

\begin{definition}
The $f$-vector of a $d-1$-dimensional simplicial complex $\Delta$ is
$f=(f_0,\ldots , f_d)$ where $f_i$ is the number of faces of $\Delta$ with dimension $i-1$. The $h$-vector $h=(h_0, \ldots , h_d)$ is defined by
\[f_0(q - 1)^d + f_1(q - 1)^{d-1} + \cdots + f_d(q - 1)^0 = h_0q^d + h_1q^{d-1} + \cdots + h_dq^0.\]
The $h$-polynomial of $\Delta$ is $h_{\Delta}(q) = \sum_i h_i q^{d-i}$.
\end{definition}

If $\Delta$ is a shellable simplicial complex then the $h$-vector can also be obtained from the restriction sets of the shelling order.
\begin{proposition} If $\Delta$ is a pure simplicial complex with shelling order $F_1,\ldots,F_s$ then $h_i = \# \{R_j : |R_j| = i \}$.
\end{proposition}

\subsection{$h$-complexes and $H$-shellings}\label{sec:hshelling}
Let $\Delta$ be a shellable simplicial complex with vertex set $E$ whose facets $F_1,F_2,\dots,F_s$ are listed here in a shelling order. Let $R_1,R_2,\dots,R_s$ be the restriction sets for this shelling. Assume that the collection $\Gamma = \{R_1,\dots,R_s\}$ is, itself, a simplicial complex. This following situation, improbable as it may seem, is known to arise in several contexts and was studied by Edelman and Reiner \cite{ER}. In this situation, the shelling of $\Delta$ is called an $h$-shelling and $\Gamma$ is called an $h$-complex.

The following result describes the $f$-vector of an $h$-complex.
\begin{proposition}
    If $\Delta$ has an $h$-shelling with corresponding $h$-complex $\Gamma$ then the $h$-vector of $\Delta$ is the $f$-vector of $\Gamma$.
\end{proposition}

Edelman and Reiner say that the shelling $F_1,\dots,F_s$ has property $(H)$ if whenever $e \in E$ and $G \in \Delta$ is a face contained in a facet $F$, the following implication holds, 
\begin{equation}\label{eq:H'}
  \tag{$H$}  \textup{if $e \in G \subset F$ and $e \in R(F)$, then $e \in R(G)$.}
\end{equation}
Here we employ the convention that for a facet $F_i$, $R(F_i) = R_i$, and for a non-facet $G$, if $G \in [R_i,F_i]$ then $R(G) = R_i$. In verifying \eqref{eq:H'}, it is sufficient to take $G$ to be codimension one in $F$ by \cite[Theorem~2.6]{ER}. (Technically our property \eqref{eq:H'} is property $(H')$ in \cite{ER}, but since these are equivalent we proceed with the un-primed notation.)
\begin{theorem}[{\cite[Theorem~2.7]{ER}}]
    If a shelling has property \eqref{eq:H'} then it is an $h$-shelling.
\end{theorem}
We call such a shelling an $H$-shelling. Property $(H)$ is strictly stronger than being an $h$-shelling; see \cite[Figure~1]{ER}. In the cases we encounter, it will be immediate that the shellings we produce are $h$-shellings since the associated collections of restriction sets will be familiar complexes. We will also show that our shellings satisfy the stronger property \eqref{eq:H'}

\section{The External Activity Complex}\label{sec:acs}
In this section we review the construction and results of Ardila, Castillo and Samper \cite{ACS}. All the results in this section are attributable to this work.

Let $E$ be a set with $n$ elements and define $E(x,z):= \{ x_e,z_e : e \in E\}$. We will employ the following notation: For subsets $S,T \subseteq E$, $x_S$ denotes $\{x_i : i \in S\}$ and $z_T$ is defined similarly. We will write $x_S z_T$ for the union $x_S \cup z_T$.

\begin{definition}
Let $M = (E, \mathcal{B})$ be a matroid with a fixed linear order $<$ on $E$. The \emph{external activity complex}, denoted $\underline{\Delta_M}$, is the simplicial complex with ground set $E(x,z)$ and facets
\[F(B):= x_{B \cup \EP(B)} z_{B \cup \EA(B)}\]
for each basis $B$.
\end{definition}
It is immediate that $\underline{\Delta_M}$ is pure of dimension $n+r-1$, when $M$ is a matroid of rank $r$ on $n$ elements.

\begin{example} Here are the bases and corresponding facets of the external activity complex in our running example.
\[
\begin{array}{c|ccc}
B& \EA(B) & \EP(B) & F(B)\\
\hline
345 & \0 & 12 &  x_{12345}z_{345}\\
135 & \0 & 24 & x_{12345}z_{135}\\
245 & \0 & 13  & x_{12345}z_{245}\\
235 & \0 & 14  & x_{12345}z_{235}\\
125 & 3 & 4  & x_{1245}z_{1235}\\
134 & 5 & 2  & x_{1234}z_{1345}\\
234 & 5 & 1 & x_{1234}z_{2345}\\
124 & 35 & \0 & x_{124}z_{12345}\\
\end{array}
\]
\end{example}

\subsection{Shellings of the external activity complex}
In this section we describe a large set of shelling orders of the external activity complex $\underline{\Delta_M}$. We will later extend these results to the augmented external activity complex. 

\begin{theorem}[{\cite[Theorem~1.1]{ACS}}]\label{thm:acs} Let $M$ be a matroid with linear order on the ground set, $E$. For any linear extension of Las Vergnas's external/internal order $\leq_{ext/int}$ on the collection of bases of $M$, the corresponding ordering of the facets is a shelling order of the external activity complex, $\underline{\Delta_M}$.
\end{theorem}

Before we review the proof of this result, it is worthwhile to note that the facets of the external activity complex are defined by the bases of a matroid along with their corresponding external activities. However, it is asserted that any linear extension of Las Vergnas's external/internal ordering of the bases produces a shelling order of this complex. This result cannot be relaxed to Las Vergnas's external order or internal order alone. Which is to say, linear extensions of $\leq_{ext}$ on the bases of $M$ and linear extensions of $\leq_{int}$ on the bases of $M$ do not necessarily produce shelling orders of $\underline{\Delta_M}$. See \cite[Examples~3.1,3.2]{ACS}. Our point here is that internal activity is a crucial feature of the proof of their result, but this is not reflected in the definition of external activity complex itself. Our augmented external activity complex has both external and internal activity used in the definition of the complex, which is a natural reflection of the external/internal order used in the proof of \cref{thm:acs}.

The proof of \cref{thm:acs} is organized as follows. Let $\prec$ be an arbitrary linear order on the collection of bases of a matroid and consider the corresponding ordering of the facets of $\underline{\Delta_M}$. Recall by \cref{def:shellable}, $\prec$ induces a shelling of the external activity complex if for every $A \prec C$ there exists $B \prec C$ and $c\in E$ so that 
\[F(A) \cap F(C) \subseteq F(B) \cap F(C) = F(C) \setminus c^{xz}\]
where $c^{xz}$ denotes one of $x_c$ or $z_c$. The following result equivalently characterizes shelling orders of the external activity complex.

\begin{lemma}[{\cite[Lemma~4.2]{ACS}}]\label{lem:acs0} Let $\prec$ be a linear order on the collection of bases a matroid. Then $\prec$ induces a shelling of the external activity complex $\underline{\Delta_M}$, if and only if for any bases $A \prec C$ there exists a basis $B \prec C$ so that
\begin{enumerate}
\item $B = C\setminus c\cup b$ where $b \neq c$,
\item $c\notin A$ and $c \in \EA(B)$ if and only if $c \in \EA(A)$, and
\item for any $d \notin B \cup C$ we have $d \in \EA(B)$ if and only if $d \in \EA(C)$.
\end{enumerate}
\end{lemma} 

It is then shown that any linear extension of Las Vergnas's external/internal order on the collection of bases of $M$ satisfies \cref{lem:acs0} and this proves \cref{thm:acs}. In particular, let $\prec$ be a linear extension of Las Vergnas's external/internal order on the collection of bases of $M$ and let $A \prec C$. This implies that $C\not\leq_{ext/int}A$. Ardila, Castillo and Samper produce a basis $B$ with $B\leq_{ext/int}C$, so that $B$ precedes $C$ in $\prec$, and then they carefully show that $B$ satisfies the conditions of \cref{lem:acs0}. 

In the following lemma we summarize the results that we will make use of and extend towards our discussion of the augmented external activity complex and its shellability. 

\begin{lemma}\label{lem:acs} For any bases $C\not\leq_{ext/int}A$ there exists a basis $B$ so that
\begin{enumerate}
\item $B<_{ext/int}C$,
\item $B = C\setminus c \cup b$ where $b \neq c$,
\item $c\notin A$ and $c \in \EA(B)$ if and only if $c \in \EA(A)$, and
\item for any $d \notin B \cup C$ we have $d \in \EA(B)$ if and only if $d \in \EA(C)$.
\item Furthermore, $c\in \IP(C)\cap\EP(A)\cap\EP(B)$, and
\item $F(A)\cap F(C) \subset F(B)\cap F(C) = F(C)\setminus z_c$.
\end{enumerate}
\end{lemma} 
\begin{proof}
    The first four claims occur in \cite{ACS} when showing that any linear extension of $\leq_{ext/int}$ on the bases of $M$ satisfies \cite[Lemma~4.2]{ACS}. 
    Claim (5) follows by the construction of $B=C\setminus c \cup b$ in the proof of \cite[Theorem~1.1]{ACS} and then applying (3). To be specific, $c\in C$ is chosen so that $c\in \IP(C)\cap \EP(A)$ and (3) implies that $c\in \EP(B)$. 
    The last claim is justified as follows: claims (2), (3) and (4) imply that $F(A)\cap F(C) \subset F(B)\cap F(C) = F(C)\setminus c^{xz}$ where $c^{xz}$ denotes one of $x_c$ or $z_c$ and (5) has that $c\in \EP(B)$. We conclude that $c\notin \EA(B)\cup B$ and this means $z_c$ does not appear in $F(B)$, giving (6).
\end{proof}

\section{The Augmented External Activity Complex}\label{sec:augea}
In this section we define the augmented external activity complex and prove our main result on its shellability. Before we give this rather technical proof, we include a fully worked example which will explain the structure of the proof.
\subsection{Basic properties and statement of the main theorem}

We will need the following notation.
Let $E$ be a finite set and define $E(x,y,z):= \{ x_e,y_e,z_e : e \in E\}$. We write $E(x)$, $E(y,z)$, etc. for the obvious subsets of $E(x,y,z)$. For subsets $S,T,U \subseteq E$, $x_S$ denotes $\{x_i : i \in S\}$, and $z_T$, $y_U$ are defined similarly. We will write $x_S y_U z_T$ for the union $x_S \cup y_U \cup z_T$. We now define our complex.

\begin{definition}
Let $M = (E, \mathcal{I})$ be a matroid and let $<$ be a linear order on $E$. Define a simplicial complex $\Delta_M$ with ground set $E(x,y,z)$ and facets
\[F(I):= x_{I \cup \EP(I)} y_Y z_{I \cup \EA(I)}\]
for every independent set $I$ where
$I = B\setminus Y$, with $B$ a basis and $Y \subset \IA(B)$ (recall that every $I$ is uniquely expressible in this way).  We call $\Delta_M$ the \emph{augmented external activity complex} of $M$.
\end{definition}
We emphasize that this definition is new. It is immediate that $\Delta_M$ is pure of dimension $n+r-1$, when $M$ is a matroid of rank $r$ on $n$ elements. Observe that the external activity complex, $\underline{\Delta_M}$, is the subcomplex of the augmented external activity complex $\Delta_M$ generated by the facets $F(B)$ where $B$ is a basis of $M$. Also, we have $\Delta_M \cap E(x,z) = \underline{\Delta_M}$, which is to say, that $\underline{\Delta_M}$ is the subcomplex of faces that contain no $y_e$, $e \in E$.

In what follows we will write $\supp_x(G)$ for $\{e \in E: x_e \in G\}$ where $G \in \Delta_M$. We call this the $x$-support of the face $G$. Analogous notation will be used for $y$ and $z$.
\begin{proposition}
   Each basis $B$ of $M$ has $\supp_y(F(B)) = \emptyset$.
\end{proposition}  

\begin{proposition}\label{prop:Suppx}
Suppose that $I$ is internally related to a basis $B$, then $\supp_x(F(I)) = \supp_x(F(B))$.
\end{proposition}
\begin{proof}
We need to show that $I\cup \EP(I) = B\cup \EP(B)$. Since $I$ is internally related to the basis $B$ we have $\EA(I)=\EA(B)$ by \cref{prop:lvprop}. Taking the complement in $E$ gives $I\cup \EP(I) = B\cup \EP(B)$.
\end{proof}


\begin{proposition}\label{prop:Suppz} Suppose that $I=B\setminus Y$ is internally related to $B$, then $\supp_z(F(I))=\supp_z(F(B))\setminus Y$.
\end{proposition}

\begin{proof}
Suppose that $I$ is internally related to the basis $B$ so that $I=B\setminus Y$ with $Y\subset \IA(B)$. We make note that $\EA(I)=\EA(B)$ by \cref{prop:lvprop} and compute,
\[\supp_z(F(I))= I \cup \EA(I) = (B\setminus Y) \cup \EA(B) = (B\cup \EA(B))\setminus Y = \supp_z(F(B))\setminus Y.\qedhere\]
\end{proof}

\begin{corollary} Suppose that $I=B\setminus Y$ is internally related to the basis $B$. The facet $F(I)$ is given by taking the facet $F(B)$ and replacing $z_Y$ with $y_Y$. That is, $F(I)= F(B)\setminus z_Y \cup y_Y$.
\end{corollary}

\begin{proof}
This is an immediate consequence of the definition along 
 with \cref{prop:Suppx} and \cref{prop:Suppz}.
\end{proof}
\begin{example}
Consider, from our running example, the basis $B=245$, which has $\IA(B)= 45$, with its internally related independent sets. The corresponding facets are listed below.
\[
\begin{array}{ccc}
I & B\setminus Y_I & F(I)\\
\hline
245 & 245\setminus \emptyset & x_{12345}z_{245}\\
25 & 245\setminus 4 & x_{12345}y_{4}z_{25}\\
24  & 245\setminus 5 & x_{12345}y_{5}z_{24}\\
2 & 245\setminus 45 & x_{12345}y_{45}z_{2}
\end{array}
\]
\end{example}

We are now ready to state our main result.
\begin{theorem}\label{thm:shelling}
Let $M$ be a matroid with linear order $<$ on the ground set, $E$. For any linear extension of Las Vergnas's external/internal order on the independent sets of $M$, the corresponding ordering of the facets of $\Delta_M$ is a shelling order.
\end{theorem}

The proof of this theorem will occupy the rest of this section. We outline the proof with an example in \cref{subsec:mainThmOutline} and then we will provide all the details in \cref{subsec:mainThmProof}. 

\subsection{Proof outline and example}\label{subsec:mainThmOutline}
Let $\prec$ be a linear extension of Las Vergnas's external/internal order on the collection of independent sets of $M$. Recall, by \cref{def:shellable}, $\prec$ gives a shelling order of the augmented external activity complex if for independent sets $I \prec K$ there exists an independent set $J \prec K$ and $c\in E$ so that 
\begin{align*}
    F(I) \cap F(K) \subset F(J) \cap F(K) = F(K) \setminus c^{xyz}
\end{align*}
where $c^{xyz}$ denotes one of $x_c$, $y_c$ or $z_c$. In our proof, given any $I \prec K$ we will produce an independent set $J \prec K$ and $c\in E$ satisfying,
\begin{align}
    \label{eq:shellingPropertyInd}\tag{$*$}
    F(I) \cap F(K) \subset F(J) \cap F(K) = F(K) \setminus z_c.
\end{align}

To do this, we will need to consider two cases. The first case occurs when $I$ and $K$ are internally related and the second case is when they are not. We now provide an example for each case. We begin with the case that $I$ and $K$ are internally related. 

Consider the linear extension of $\leq_{ext/int}$ on the independent sets from our running example indicated below.  
\[ \begin{array}{ccccccccccc}
 & \hspace{-1em}I &  & & &  &\hspace{-1em} K &  &  &  & \\
\emptyset \prec & 3 \prec& 4 \prec& 5\prec & 34 \prec& 35\prec & 45\prec & 345\prec & \cdots \prec & 124
\end{array}
\]
The sets $I = 3$ and $K = 45$ (suppressing brackets) are indicated here with $I \prec K$. These sets are internally related to the basis $C=345$, which has $\IA(C)=345$. See that,
\[ \begin{array}{ccccccc}
I &=& C\setminus Y_I, &  & K & = & C\setminus Y_K,\\
3 &=& 345\setminus 45, &  & 45 & = & 345\setminus 3.
\end{array}
\]
Since $I\prec K$ are internally related, this implies that $K\not\leq_{ext/int}I$ and $K\not\subset I$. Choose $c\in K\setminus I$ and notice that $c\in \IA(C)$. 

Now we define $J:=K\setminus c$ and argue that this is the requisite independent set $J$ in \eqref{eq:shellingPropertyInd}. In our example, we see that $K\setminus I=45$ and we will choose $c=4$, which is in $\IA(C) = 345$. We have,
\[
\begin{array}{ccc}
J&:= &K\setminus c \\
5 &:= &45\setminus 4.
\end{array}
\]
Here is a figure showing $I$, $J$ and $K$ in the external/internal order on independent sets (only part of the poset is shown here).
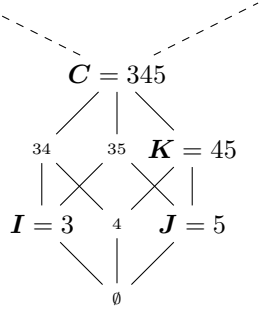
\[
\begin{tikzpicture}[scale=2]
\node (345) at (0,1) { $\boldsymbol{C}=345$};
\node (45) at (.5,.5) {$\boldsymbol{K}=45$};
\node (35) at (0,.5) {\tiny $35$};
\node (34) at (-.5,.5) {\tiny $34$};
\node (5) at (.5,0) {$\boldsymbol{J}=5$};
\node (4) at (0,0) {\tiny $4$};
\node (3) at (-.5,0) {$\boldsymbol{I}=3$};
\node (0) at (0,-.5) {\tiny $\emptyset$};
\draw (0)--(3)--(34)--(345);
\draw (0)--(5)--(45)--(345);
\draw (3)--(35);
\draw (4)--(34);
\draw (4)--(45);
\draw (5)--(35);
\draw (0)--(4);
\draw (35)--(345);
\node (1) at (-1,1.5) {};
\node (2) at (1,1.5) {};
\draw[dashed] (345)--(1);
\draw[dashed] (345)--(2);
\end{tikzpicture}
\]
Recall that the facet corresponding to the basis $C$ is $F(C)=x_{12345}z_{345}$. For the independent sets $I$, $J$ and $K$ internally related to $C$ we have the corresponding facets, 
\[
\begin{array}{ccc}
    F(I)=x_{12345}y_{45}z_{3}, & F(J)=x_{12345}y_{34}z_{5}, & F(K)=x_{12345}y_{3}z_{45}.
\end{array}
\]
Comparing the facets gives the desired result:
\[
\begin{array}{ccccc}
    F(I)\cap F(K) &\subset &F(J)\cap F(K) &= &F(K)\setminus z_c  \\
    x_{12345} &\subset &x_{12345}y_{3}z_{5} &= &F(45)\setminus z_4.
\end{array}
\]
Summarizing, given $I \prec K$ that are internally related to $C$, we will show that $K \not\subset I$ and that there is $c \in K\setminus I \subset  \IA(C)$. We will then prove that $J = K \setminus c$ is a set and $c$ is an element that satisfies \eqref{eq:shellingPropertyInd}.

We proceed to the case that $I$ and $K$ are not internally related. Consider the following linear extension of $\leq_{ext/int}$ on the independent sets from our running example shown below.
\[ \begin{array}{ccccccccccccc}
  &  & &  &  &  &  & \hspace{-1em}I &  & \hspace{-1em}K &  & \\
 \cdots \prec & 345 \prec &\cdots \prec& 245 \prec& \cdots \prec & 135 \prec& \cdots \prec & 23 \prec& \cdots\prec & 14 \prec& \cdots \prec& 124
\end{array}
\]

The sets $I=23$ and $K=14$ are indicated here with $I\prec K$. The set $I$ is internally related to the basis $A=235$, which has $\IA(A)=5$ and the set $K$ is internally related to the basis $C=134$, which has $\IA(C)=3$. See that,

\[ \begin{array}{cccccc}
I & = A\setminus Y_A, &  & K & = & C\setminus Y,\\
23 & = 235\setminus 5, &  & 14 & = & 134\setminus 3.
\end{array}
\]

To produce the required independent set $J$ that satisfies \eqref{eq:shellingPropertyInd}, we first consider the bases $A$ and $C$ and choose a basis $B$ that satisfies \cref{lem:acs}. We then construct the independent set $J$ from the basis $B$ by deleting a subset of its internally active elements.

Observe that since $I$ and $K$ are independent sets related to bases $A$ and $C$, respectively with $K\not\leq_{ext/int} I$ then $C \not\leq_{ext/int} A$. Choose a basis $B=C\setminus c \cup b$ that satisfies \cref{lem:acs}. In this example, we choose $c=4$ and $b=5$ so that,
\begin{align*}
B &=C\setminus c \cup b\\
135 & = 134 \setminus 4 \cup 5.
\end{align*}
Here is a figure showing $A$, $B$ and $C$ in the external/internal order on bases. 
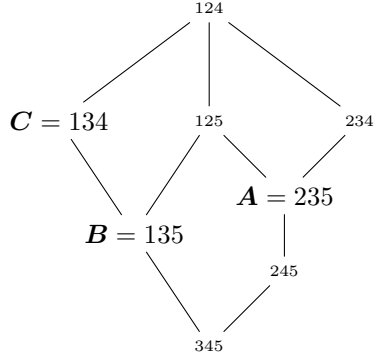
\[
\begin{tikzpicture}[scale=1]
\node (0) at (0,0) {\tiny $345$};
\node (1) at (-1,1.5) {$\boldsymbol{B} = 135$};
\node (2) at (1,1) {\tiny $245$};
\node (23) at (1,2) {$\boldsymbol{A}=235$};
\node (123) at (0,3) {\tiny $125$};
\node (145) at (-2,3) {$\boldsymbol{C}=134$};
\node (2345) at (2,3) {\tiny $234$};
\node (12345) at (0,4.5) {\tiny $124$};
\draw (0)--(1)--(123)--(12345);
\draw (1)--(145)--(12345);
\draw (0)--(2)--(23)--(123);
\draw (23)--(2345)--(12345);
\end{tikzpicture}
\]
Recall for the bases $A$, $B$ and $C$, we have the corresponding facets,
\[
\begin{array}{ccc}
    F(A) = x_{12345}z_{235}, & F(B)=x_{12345}z_{135}, & F(C)=x_{1234}z_{1345}.
\end{array}
\]
Comparing the facets for these bases gives,
\[
\begin{array}{ccccc}
    F(A)\cap F(C) &\subset &F(B)\cap F(C) &= &F(C)\setminus z_c  \\
    x_{1234}z_{35} &\subset &x_{1234}z_{135} &= &F(134)\setminus z_4.
\end{array}
\]
We now construct the set $J$ from the basis $B$. We claim that the basis $B$ produced necessarily has $Y\subset \IA(C)\subset \IA(B)$. In this example, we observe that,
\[ 
\begin{array}{ccccc}
    Y &\subset & \IA(C) &\subset &\IA(B)\\ 
   3 &\subset &3 &\subset &35.
\end{array}
\]
The above inclusion allows us to define $J:=B\setminus Y$ so that $J$ is internally related to $B$. We argue that $J$ is the desired independent set that is a witness to \eqref{eq:shellingPropertyInd}. In our example, we have
\begin{align*}
J&:=B\setminus Y\\
15 &:= 135\setminus 3.
\end{align*}
Here is a figure showing $I$, $J$ and $K$ in the external/internal order on independent sets.
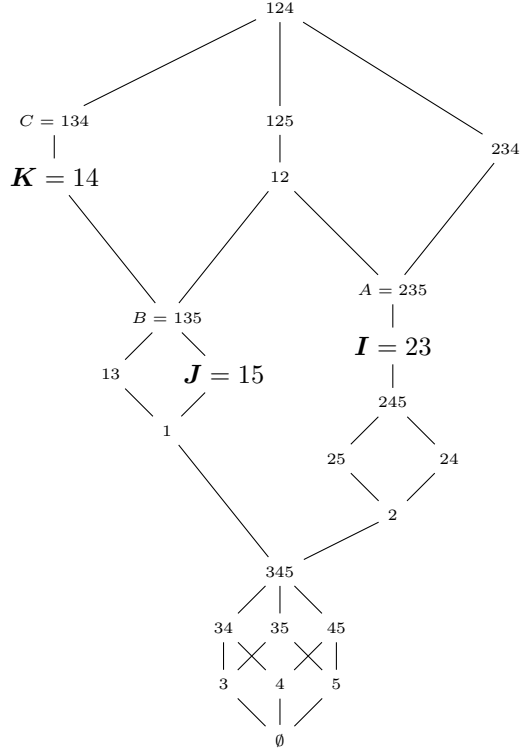
\[
\begin{tikzpicture}[scale=1.5]
\node (345) at (0,1) {\tiny $345$};
\node (45) at (.5,.5) {\tiny $45$};
\node (35) at (0,.5) {\tiny $35$};
\node (34) at (-.5,.5) {\tiny $34$};
\node (5) at (.5,0) {\tiny $5$};
\node (4) at (0,0) {\tiny $4$};
\node (3) at (-.5,0) {\tiny $3$};
\node (0) at (0,-.5) {\tiny $\emptyset$};
\draw (0)--(3)--(34)--(345);
\draw (0)--(5)--(45)--(345);
\draw (3)--(35);
\draw (4)--(34);
\draw (4)--(45);
\draw (5)--(35);
\draw (0)--(4);
\draw (35)--(345);
\node (135) at (-1,3.25) {\tiny $B=135$};
\node (13) at (-1.5,2.75) {\tiny $13$};
\node (15) at (-.5,2.75) {$\boldsymbol{J}=15$};
\node (1) at (-1,2.25) {\tiny $1$};
\draw (1)--(13)--(135);
\draw (1)--(15)--(135);
\node (2) at (1,1.5) {\tiny $2$};
\node (25) at (.5,2) {\tiny $25$};
\node (24) at (1.5,2) {\tiny $24$};
\node (245) at (1,2.5) {\tiny $245$};
\draw (2)--(25)--(245);
\draw (2)--(24)--(245);
\node (235) at (1,3.5) {\tiny $A=235$};
\node (23) at (1,3) {$\boldsymbol{I}=23$};
\draw (23)--(235);
\node (134) at (-2,5) {\tiny $C=134$};
\node (14) at (-2,4.5) {$\boldsymbol{K}=14$};
\draw (14)--(134);
\node (125) at (0,5) {\tiny $125$};
\node (12) at (0,4.5) {\tiny $12$};
\draw (12)--(125);
\node (234) at (2,4.75) {\tiny $234$};
\node (124) at (0,6) {\tiny $124$};
\draw (345)--(1);
\draw (345)--(2);
\draw (245)--(23);
\draw (135)--(14);
\draw (135)--(12);
\draw (235)--(12);
\draw (235)--(234);
\draw (134)--(124);
\draw (125)--(124);
\draw (234)--(124);
\end{tikzpicture}
\]

For the independent sets $I$, $J$ and $K$ internally related to $A$, $B$ and $C$, respectively, recall that the corresponding facets of ${\Delta_M}$ are obtained by migrating a subset of internally active $z$'s to $y$'s. In our example we have,
\[
\begin{array}{ccc}
    F(I)=x_{12345}y_{5}z_{23}, & F(J)=x_{12345}y_{3}z_{15}, & F(K)=x_{1234}y_{3}z_{145}.
\end{array}
\]
Comparing the facets gives the following desired result:
\[
\begin{array}{ccccc}
    F(I)\cap F(K) &\subset &F(J)\cap F(K) &= &F(K)\setminus z_c , \\
    x_{1234} &\subset &x_{1234}y_{3}z_{15} &= &F(14)\setminus z_4.
\end{array}
\]

Summarizing, given $I \prec K$ that are internally related to the bases $A$ and $C$, we will show that $C \not\leq_{ext/int} A$ and choose a basis $B=C \setminus c \cup b$ that satisfies \cref{lem:acs}. Further, given that $K=C\setminus Y$ where $Y\subset \IA(C)$, we will show that $Y\subset \IA(C) \subset \IA(B)$. We will then prove that $J = B\setminus Y$ is an independent set and $c$ is an element that satisfies \eqref{eq:shellingPropertyInd}.

\subsection{Proof of the main theorem}\label{subsec:mainThmProof}
Fix a linear extension $\prec$ of the order $\leq_{ext/int}$ on $\mathcal{I}(M)$. Assume that $I \prec K$, equivalently, $K \not\leq_{ext/int} I$. We will produce an independent set $J$ and element $c$ such that $J<_{ext/int}K$ and \eqref{eq:shellingPropertyInd} holds.

Mirroring our example, we begin with case that
$I$ and $K$ are both internally related to the basis $C$.

\begin{lemma}\label{lem:case1constructJ} Suppose that $I$ and $K$ are independent sets internally related to a basis $C$ with $K\not\leq_{ext/int} I$. Then 
\begin{enumerate}
   \item $K\setminus I \subset \IA(C)$,
    \item there exists $c\in K\setminus I$, and
    \item the independent set $J:=K\setminus c$ is internally related to $C$ with $J<_{ext/int}K$.
    \item Also, $c\in \EP(I)$.
\end{enumerate}
\end{lemma}

\begin{proof}
    Since $I$ and $K$ are internally related to the basis $C$ then $C$ contains $K$ and $I$ contains $\IP(C)$ by \cref{prop:crapo}(2). It follows that $K\setminus I\subset C\setminus \IP(C) = \IA(C)$.
    The independent sets $I$ and $K$ are internally related with $K\not\leq_{ext/int}I$ and so by definition, $K\not\subset I$. This implies that there must exist $c\in K\setminus I\subset \IA(C)$.
    Suppose that $K=C\setminus Y$ where $Y\subset \IA(C)$. The independent set 
\begin{align*}
    J:= K\setminus c = C\setminus (Y\cup c) 
\end{align*} 
has $Y\cup c\subset \IA(C)$ and hence $J$ is internally related to the basis $C$. Moreover, $J \subset K$ so by definition, $J<_{ext/int}K$.
    Recall that $I$ is internally related to the basis $C$ so by \cref{prop:Suppx}, $I\cup \EP(I)=C\cup \EP(C)$. However, $c\in C$ and $c\notin I$ and thus $c\in \EP(I)$.
\end{proof}

\begin{lemma}\label{lem:case1eq} Suppose that $I$ and $K$ are independent sets internally related to a basis $C$ with $K\not\leq_{ext/int} I$. Let $J$  and $c \in K$ be an independent set and element as furnished by \cref{lem:case1constructJ}. Then,
\begin{enumerate}
    \item $J <_{ext/int} K$; in particular $J$ precedes $K$ in any linear extension of $\leq_{ext/int}$.
    \item $F(I)\cap F(K) \subset F(J)\cap F(K) =  F(K)\setminus z_c$.
\end{enumerate}
\end{lemma}
\begin{proof}

The first claim is a restatement of \cref{lem:case1constructJ}(3) and the definition of being a linear extension.

For the second claim, we compare the supports of the corresponding facets. For the $x$-support, recall that $I$, $J$ and $K$ are all internally related to the basis $C$. By \cref{prop:Suppx}, $\supp_x(F(I))= \supp_x(F(J))= \supp_x(F(K))$ and hence,
    \begin{align*}
        \supp_x(F(I))\cap\supp_x(F(K)) = \supp_x(F(J))\cap \supp_x(F(K)) =\supp_x(F(K)).
    \end{align*}    
    
    We now consider the $y$-support of these facets, by first comparing the facets corresponding to $J$ and $K$. Suppose that $K=C\setminus Y$ and $J=C\setminus (Y\cup c)$ with $Y\cup c\subset \IA(C)$. We have,
    \begin{align*}
        \supp_y(F(J))\cap \supp_y(F(K)) &= (Y\cup c)\cap Y
        = Y
        = \supp_y(F(K)).
    \end{align*}
    This implies,
    \begin{align*}
        \supp_y(F(I))\cap\supp_y(F(K)) \subset \supp_y(F(K)) = \supp_y(F(J))\cap \supp_y(F(K)).
    \end{align*}
    
    For the $z$-support, again, we first compare the facets corresponding to $J$ and $K$ and note that $c\notin \EA(C)$ since $c \in K\setminus I \subset C$. We compute, 
    \begin{align*}
        \supp_z(F(J))\cap \supp_z(F(K)) &= (J \cup \EA(J)) \cap (K \cup \EA(K))\\
        &= (J \cap K) \cup \EA(C) \quad \textup{($\EA(J) = \EA(K) = \EA(C)$ by \cref{prop:lvprop})}\\
        &= (K\setminus c) \cup \EA(C)\\
        &=(K \cup \EA(K))\setminus c\\ 
        &= \supp_z(F(K))\setminus c.
    \end{align*}
    To include the facet corresponding to $I$ in our comparison, we recall that $c\in \EP(I)$  by \cref{lem:case1constructJ}(4). This implies that $c\notin I \cup \EA(I) =\supp_z(F(I))$. We conclude that
    \begin{align*}
        \supp_z(F(I))\cap\supp_z(F(K)) \subset \supp_z(F(K))\setminus c = \supp_z(F(J))\cap \supp_z(F(K)).
    \end{align*}
    This completes the proof of our lemma.
\end{proof}
This lemma proves that the required shelling property holds when $I$ and $K$ are internally related. We now consider the case that $I$ and $K$ are not internally related. We will take $I$ and $K$ related to bases $A$ and $C$, respectively, in what follows. The argument then proceeds in three lemmas: The first lemma builds on the results of \cite{ACS} discussed in  \cref{sec:acs}. The second lemma uses the first to define the independent set $J$ and element $c$ that verifies the shelling property. Finally, the third lemma proves that the proposed $J$ and $c$ defined satisfy the required shelling property and this completes the proof of our main theorem. 

\begin{lemma}\label{lem:acs2.0} Suppose that $A$ and $C$ are bases with $C\not\leq_{ext/int}A$ and produce a basis $B\leq_{ext/int}C$ with $B=C\setminus c\cup b$ as in \cref{lem:acs}. Then $\IA(C)\subset \IA(B)$.
\end{lemma}

\begin{proof} 
Let us begin by defining $X:=C\setminus c$ so that $C=X\cup c$ and $B=X\cup b$. Since $B\leq_{ext/int}C$ by \cref{LVextint} we have, 
\begin{align*}
    \IP(B)\cup \EA(B) \subset \IP(C)\cup \EA(C).
    \end{align*}
Recall, that $c\in \EP(B)$ by \cref{lem:acs}(5). This implies that $c\notin \IP(B)\cup \EA(B)$ and thus,
\begin{align*}
    \IP(B)\cup \EA(B) &\subset (X\cap \IP(C)) \cup \EA(C).
\end{align*}
Further, \cref{lem:acs}(4) yields, $\EA(C)\subset \EA(B)\cup b$ and this implies,
\begin{align*}
    \IP(B)\cup \EA(B) &\subset (X\cap \IP(C))\cup \EA(B)\cup b.
\end{align*}
Next, we will consider two cases according to the internal activity of $b\in B$. If it is the case that $b\in \IP(B)$ then,
\begin{align*}
    (X\cap \IP(B)) \cup \EA(B) \cup b &\subset (X\cap \IP(C)) \cup \EA(B) \cup b.
\end{align*}
We conclude that $X\cap \IP(B)\subset X\cap \IP(C)$. If it is the case that $b\notin \IP(B)$ then $\IP(B) = X\cap \IP(B)$ and so,
\begin{align*}
    (X\cap \IP(B)) \cup \EA(B) &\subset (X\cap \IP(C)) \cup \EA(B).
\end{align*}
Again, we conclude that $X\cap \IP(B)\subset X\cap \IP(C)$. In both cases, taking the complement in $X$, gives 
\begin{align*}
    X\cap \IA(C) \subset X\cap \IA(B) \subset \IA(B).
\end{align*}
Recall, $c\in \IP(C)$ by \cref{lem:acs}(5) and so $X \cap \IA(C)=\IA(C)$ and the result follows.
\end{proof}
We are now ready to construct the needed independent set $J$.
\begin{lemma}\label{lem:case2constructJ}
Suppose that the independent sets $K\not\leq_{ext/int}I$ are not internally related. Further suppose that $A$ and $C$ are the distinct bases internally related to $I$ and $K$, respectively. Write $K = C \setminus Y$ where $Y \subset \IA(C)$. 

Observe that $C\not\leq_{ext/int}A$ and let $B<_{ext/int} C$ and $c$ be a basis and element produced as in \cref{lem:acs} from $A$ and $C$. We have,
\begin{enumerate}
\item $Y\subset \IA(B)$,
\item the independent set $J:=B\setminus Y$ is internally related to $B$,
\item $J<_{ext/int}K$, and
\item $c\in \EP(I)$.
\end{enumerate}
\end{lemma}
\begin{proof}
The independent set $K=C\setminus Y$ has $Y\subset \IA(C)$ and applying \cref{lem:acs2.0} gives the desired inclusion, $Y\subset \IA(C) \subset \IA(B)$.

By definition, $J:=B\setminus Y$ where $Y\subset \IA(B)$ by (1) and therefore $J$ is an independent set internally related to $B$.

We have that the independent sets $J$ and $K$ are internally related to the bases $B$ and $C$, respectively, with $B<_{ext/int}C$ and this gives $J<_{ext/int}K$.

Recall that $c\in \EP(A)$ by \cref{lem:acs}(5) and $I$ is internally related to $A$. \cref{prop:Suppx} yields $I\cup \EP(I)=A\cup \EP(A)$. However, $c\notin A\supset I$ and therefore $c\in \EP(I)$. 
\end{proof}
We conclude by proving that the $J$ produced in the previous lemma satisfies \eqref{eq:shellingPropertyInd}.
\begin{lemma}\label{lem:case2eq} 
Maintaining the notation from \cref{lem:case2constructJ}, we have
\begin{enumerate}
    \item $J <_{ext/int} K$; in particular $J$ precedes $K$ in any linear extension of $\leq_{ext/int}$.
    \item $F(I)\cap F(K) \subset F(J)\cap F(K) =  F(K)\setminus z_c$.
\end{enumerate}
\end{lemma}

\begin{proof}
The first claim is a restatement of \cref{lem:case2constructJ}(3) and the definition of being a linear extension.

For the second claim, we compare the supports of the facets $F(I)$,  $F(J)$ and $F(K)$. For the $x$-support, first observe that by \cref{lem:acs}(6), the facets corresponding to the bases $A$, $B$ and $C$ satisfy 
\begin{align*}
    \supp_x(F(A))\cap\supp_x(F(C)) \subset \supp_x(F(B))\cap \supp_x(F(C)) =\supp_x(F(C)).
\end{align*}
Since $I$, $J$ and $K$ are internally related to the bases $A$, $B$ and $C$, respectively, by \cref{prop:Suppx}, this gives $\supp_x(F(I))= \supp_x(F(A))$, $\supp_x(F(J))= \supp_x(F(B))$ and $\supp_x(F(K))=\supp_x(F(C))$. We conclude that,
    \begin{align*}
        \supp_x(F(I))\cap\supp_x(F(K)) \subset \supp_x(F(J))\cap \supp_x(F(K)) =\supp_x(F(K)).
    \end{align*}  
    We now consider the $y$-support of the facets $F(I)$, $F(J)$ and $F(K)$ by first comparing the facets corresponding to $J$ and $K$. By definition of $K$ and construction of $J$, 
    \[\supp_y(F(J))\cap \supp_y(F(K)) = Y = \supp_y(F(K)).
    \]
    This implies,
    \begin{align*}
        \supp_y(F(I))\cap\supp_y(F(K)) \subset \supp_y(F(K)) = \supp_y(F(J))\cap \supp_y(F(K)).
    \end{align*}

    For the $z$-support, we first observe that by \cref{lem:acs2.0}(6), the facets corresponding to the bases $A$, $B$ and $C$ satisfy 
\begin{align*}
    \supp_z(F(A))\cap\supp_z(F(C)) \subset \supp_z(F(B))\cap \supp_z(F(C)) =\supp_z(F(C))\setminus c.
\end{align*}
We wish to compare the facets corresponding to $J$ and $K$, first. Making use of \cref{prop:Suppz} we write the $z$-supports of $J$ and $K$ in terms of $z$-supports of their corresponding bases. This gives $\supp_z(F(J))= \supp_z(F(B))\setminus Y$ and $\supp_z(F(K))=\supp_z(F(C))\setminus Y$.
Noting that $c\in \IP(C)$, and $Y\subset \IA(C)$ we have the following, 
\begin{align*}
    \supp_z(F(B))\cap \supp_z(F(C)) &=\supp_z(F(C))\setminus c,\\
    (\supp_z(F(B))\cap \supp_z(F(C))) \setminus Y &=(\supp_z(F(C))\setminus c)\setminus Y,\\
    (\supp_z(F(B))\setminus Y)\cap (\supp_z(F(C)) \setminus Y) &=(\supp_z(F(C))\setminus Y)\setminus c,\\
    \supp_z(F(J))\cap \supp_z(F(K)) &=\supp_z(F(K))\setminus c.
\end{align*}
    To include the facet corresponding to $I$ in our comparison, we recall from \cref{lem:case2constructJ}(4) that $c\in \EP(I)$ and so $c\notin I \cup \EA(I) =\supp_z(F(I))$. We conclude that
    \begin{align*}
        \supp_z(F(I))\cap\supp_z(F(K)) \subset \supp_z(F(K))\setminus c = \supp_z(F(J))\cap \supp_z(F(K)).
    \end{align*}
    This completes the proof.
\end{proof}

\section{Restriction Sets}\label{sec:restriction}
Here we investigate the restriction sets of two different families of shelling orders of $\Delta_M$. In one case we find that the restriction sets form (a complex isomorphic to) the independence complex of $M$. In the second case the restriction sets give rise to a two-variable enrichment of the $h$-polynomial that has a nice expression in terms of the Tutte polynomial.
\subsection{Restriction sets of linear extensions of $\leq_{ext/int}$}
\begin{proposition}\label{prop:RestrictionSets}
Let $\prec$ be a linear extension of $\leq_{ext/int}$ on $\mathcal{I}(M)$. For each independent set $I$ of $M$, then the restriction set of $F(I)$ in this order is equal to $z_I$. 
\end{proposition}
\begin{proof}
Our argument proceeds in two parts: 
\begin{enumerate}
\item We begin by showing that $z_I$ is a new face when $F(I)$ is added to the complex generated by the facets preceding it in $\prec$. For this, we will prove that if $z_I \subseteq F(J)$ then $I \leq_{ext/int} J$ which means that $I$ must come before $J$ in $\prec$. 
\item We complete the proof by showing that $z_I$ is the minimal new face when $F(I)$ is added to the complex generated by the facets preceding it in $\prec$. To accomplish this, we argue that if $z_S \subsetneq z_I$ then $z_S$ must be contained in a facet that precedes $F(I)$ in $\prec$. 
\end{enumerate}
For (1), write $I=A\setminus Y_A$ and $J=B\setminus Y_B$ with $A$, $B$ bases, $Y_A\subset \IA(A)$ and $Y_B\subset \IA(B)$. Since $z_I\subset F(J)$ we have $I \subset J \cup \EA(J)$. If $I$ and $J$ are not related then 
\[\IP(A)= \IP(I) \subset I\subset J\cup \EA(J)\subset B\cup \EA(B).\] 
It follows that $\IP(A)\cap \EP(B) = \emptyset$ and so $A\leq_{ext/int}B$ by \cref{LVextint}(2). This implies $I\leq_{ext/int}J$. If $I$ and $J$ are related then $A=B$ and $I\subset J\cup \EA(J)$ implies that $I\subset J$, so $I\leq_{ext/int}J$ and we are done.

For (2), observe that $S\subsetneq I$ so $S$ is an independent set. Write $S=B\setminus G$ where $B$ is a basis and $G\subseteq \IA(B)$ by \cref{prop:crapo}(2) and $I=A\setminus Y$ where $A$ is a basis and $Y\subset \IA(A)$. Notice that $S = B\setminus G  \subset B \cup \EA(B)$ so that $z_S$ is contained in the facet $F(B)$. If $A\neq B$ then 
\[\IP(B)\subset B\setminus G = S \subsetneq I \subset A.\] 
By \cref{LVint}(2) this implies that $B<_{ext/int}A$. Because $I$ is internally related to $A$, we obtain $B<_{ext/int}I$. We have just shown that $z_S$ is contained in the facet $F(B)$ and $F(B)$ precedes $F(I)$ in $\prec$. It remains to consider the case $A=B$. If $A=B$ then $S=A\setminus G$ where $G\subset \IA(A)$. This gives $S\subsetneq I$ with $I$ and $S$ are both internally related to the basis $A$. Further, $z_S$ is contained in the facet $F(S)$, with $S<_{ext/int}I$ by definition and thus $F(S)$ precedes $F(I)$ in $\prec$.
\end{proof}
The following corollary is immediate from the definition of the Tutte polynomial.
\begin{corollary}
The $h$-polynomial of $\Delta_M$ is 
\[
\sum_{I \in \mathcal{I}(M)} q^{n+r-|I|} = q^n T_M(1+q,1) 
\]
where $T_M$ is the Tutte polynomial of $M$.
\end{corollary}
Recall from \cref{sec:hshelling} that a shelling  of a simplicial complex is called an $h$-shelling if its restriction sets form a subcomplex. Since the restriction sets of a shelling of $\Delta_M$ produced from any linear extension of $\leq_{ext/int}$ forms a complex visibly isomorphic to the independence complex of $M$, it is immediate that this is an $h$-shelling. However, we can say more.
\begin{corollary}\label{cor:H'}
 For any linear extension of $\leq_{ext/int}$, the corresponding shelling of $\Delta_M$ has property \eqref{eq:H'}.
\end{corollary}
This further explains why the $h$-vector of $\Delta_M$ is the $f$-vector of $\mathcal{I}(M)$.
\begin{proof}
Assume that $G \subset F(J)$ is face. It is sufficient to consider the case that $G$ is codimension one in $F(J)$, by \cite[Theorem~2.6]{ER}. Say that $G$ is obtained by deleting some $x_i$, $y_i$ with $i \in E$, or $z_i$ with $i \notin J$, from $F(J)$. Then $G \in [z_J,F(J)]$ and hence $R(G) = R(F(J))=z_J$ and the implication in \eqref{eq:H'} is trivial.

Assume $G = F(J) \setminus z_j$ with $j \in J$. Say that $G \in [z_I, F(I)]$ for some independent set $I$. To show property \eqref{eq:H'}, it is sufficient to prove that $J \setminus j \subset I$. For a facet of $\Delta_M$, we can recover its corresponding independent set by intersecting its $x$ and $z$ supports. Thus,
\[
J \setminus j = \supp_x(G) \cap \supp_z(G) \subset
\supp_x(F(I)) \cap \supp_z(F(I)) = I.\qedhere
\]
\end{proof}

\subsection{A bivariate $h$-polynomial}
Recall that $\leq_{ext/int}$ is a partial order on $\mathcal{I}(M)$ defined by saying $I \leq_{ext/int} J$ if and only if the related bases $A$ and $B$ are distinct and satisfy $A \leq_{ext/int} B$, or the bases related to $I$ and $J$ are the same and $I \subset J$. In this section we consider a different partial order on $\mathcal{I}(M)$.

For $I,J \in \mathcal{I}(M)$, say that $I \leq'_{ext/int} J$ if and only if  the related bases $A$ and $B$ are distinct and satisfy $A \leq_{ext/int} B$, or the bases related to $I$ and $J$ are the same and $J \subset I$. Informally, one obtains $\leq'_{ext/int}$ from $\leq_{ext/int}$ by flipping each small boolean lattice in $\leq_{ext/int}$ upside down; see the Hasse diagram in \cref{ex:running2} and \cref{rem:auto}.

It is a straightforward generalization of the (admittedly subtle) arguments in \cref{sec:augea} that every linear extension of $\leq'_{ext/int}$ gives a shelling of $\Delta_M$. We leave the details to the motivated reader. 

We wish to describe the restriction sets for shelling orders, $\prec'$, on $\Delta_M$ that are produced from linear extensions of $\leq'_{ext/int}$. Before we state and prove our result we will require a lemma.

\begin{lemma}\label{lem:restrictionSets'} Let $A$ be a basis with $a\in \IP(A)$ and define the basis 
\[D:=A\setminus a \cup d\] 
where $d$ is a maximum. Then $D<_{ext/int}A$ and $\IA(A)\subset \IA(D)$.
\end{lemma}
\begin{proof}
Let us begin by defining $X:=A\setminus a$ so that $A=X\cup a$ and $D=X\cup d$. Notice that $d\in \IA(D)$ and this gives $\IP(D)\subset X \subset A$. We obtain $D<_{int}A$ by \cref{LVint}(2) and hence $D<_{ext/int}A$. 

To show that $\IA(A)\subset \IA(D)$, first observe that $D<_{int}A$ and so by \cref{LVint}(3), $\IP(D)\subset \IP(A)$. Since $d\notin \IP(D)$ this gives $\IP(D)=\IP(D)\cap X$. Also, $a\notin D\supset \IP(D)$ and thus
\[\IP(D)\cap X =\IP(D) \subset \IP(A)\cap X.\]
Taking the complement in $X$ gives
\[\IA(A)\cap X \subset \IA(D)\cap X \subset \IA(D).\]
Recall, $a\in \IP(A)$ and this implies that $\IA(A)=\IA(A)\cap X$. The result follows.
\end{proof}

\begin{proposition}
 Let $\prec'$ be a linear extension of $\leq'_{ext/int}$ on $\mathcal{I}(M)$. For each independent set $I$ related to a basis $A$, write $I = A \setminus Y_I$ where $Y_I \subset \IA(A)$. Then the restriction set of $F(I)$ in $\prec'$ is equal to $y_{Y_I} z_{\IP(A)}$.
\end{proposition}

\begin{proof}
Similar to \cref{prop:RestrictionSets}, our proof is organized in two parts:
\begin{enumerate}
\item We begin by showing that $y_{Y_I}z_{\IP(A)}$ is a new face when $F(I)$ is added to the complex generated by the facets preceding it in $\prec'$.
\item We finish by showing that $y_{Y_I}z_{\IP(A)}$ is the minimal new face when $F(I)$ is added to the complex generated by the facets preceding it in $\prec'$.
\end{enumerate}

For (1), let $I=A\setminus Y_I$ and $J=B\setminus Y_J$ be independent sets internally related to the bases $A$, $B$ with $Y_I\subset \IA(A)$ and $Y_J\subset \IA(B)$. It is sufficient to show that if $y_{Y_I}z_{\IP(A)}$ is contained in a facet $F(J)$ then $I\leq'_{ext/int}J$.

Suppose that $y_{Y_I} z_{\IP(A)}\subset F(J)$. Consideration of the $z$-support gives
\[\IP(A) \subset J \cup \EA(J)\subset  B\cup \EA(B).\]
This implies that $\IP(A)\cap \EP(B) = \emptyset$ and thus $A\leq_{ext/int}B$. Since $A$ and $B$ are bases we obtain, $A\leq'_{ext/int}B$. If $A$ and $B$ are distinct bases then $I\leq'_{ext/int}J$ and we are done. Otherwise $A=B$, and we examine the $y$-supports. We see that $Y_I\subset Y_J$ and this implies that $J=A\setminus Y_J\subset A\setminus Y_I = I$. Again, we conclude that $I\leq'_{ext/int}J$.

For (2), let $I=A\setminus Y_I$ with $A$ a basis and $Y\subset \IA(I)$. It is sufficient to show that if $y_Sz_T \subsetneq y_{Y_I}z_{\IP(A)}$ then $y_Sz_T$ is contained in a facet that precedes $F(I)$ in $\prec'$.  

 Let us assume that $y_S z_T\subsetneq y_{Y_I} z_{\IP(A)}$. We will consider two cases according to the $y$-support. In both cases, we will construct an independent set $J$ so that the facet $F(J)$ contains the face $y_Sz_T$ and $J<'_{ext/int}I$. 

For the first case, say that $S\subsetneq Y_I$. Define $J:=A\setminus S$ and observe that $J$ is an independent set internally related to $A$. We claim that the face $y_Sz_T$ is contained in the facet $F(J)$. To see this, we have $S=\supp_y(F(J))$ by definition. Also, 
\[T\subset \IP(A) \subset J \subset  \supp_z(F(J)).\] 
Further, $I=A\setminus Y_I\subsetneq A\setminus S=J$ is properly contained in $J$ and thus $J<'_{ext/int}I$. 

For the second case, say that $S = Y_I$ and $T\subsetneq \IP(A)$. Let $a\in \IP(A)\setminus T$ and define the basis, 
\[D:=A\setminus a \cup d\] 
where $d$ is a maximum. The basis $D$ satisfies \cref{lem:restrictionSets'}. This gives $D<_{ext/int}A$ so that $D<'_{ext/int}A$. Also, $S = Y_I \subset \IA(A)\subset \IA(D)$. We claim that $J:=D\setminus S$ is the needed independent set. To see this, observe that it is immediate that $J<'_{ext/int}I$, so we argue that $y_Sz_T$ is contained in the facet $F(J)$. 

For the $y$-support, we have $S=\supp_y(F(J))$ by definition. For the $z$-support, by construction
\[T\subset \IP(A)\setminus a \subset A\setminus a \subset D.\]
Now recall that $Y_I\cap \IP(A) = \emptyset$, $S=Y_I$ and $T\subset \IP(A)$. This implies that $S\cap T= \emptyset$. Since $T\subset D$, we obtain the following inclusion,
\[T\subset D\setminus S = J \subset \supp_z(F(J)).\]
We have just shown that $y_Sz_T$ is contained in the facet $F(J)$ with $J<'_{ext/int}I$ and so we are done. 
\end{proof}

The restriction sets here have more information associated to them than their size alone: We may record how the restriction set are partitioned into $y$'s and $z$'s, thus obtaining a bivariate $h$-polynomial. This polynomial depends on $\leq'_{ext/int}$ but not $\prec'$.
\begin{corollary}
 Maintaining the notation above, the bivariate $h$-polynomial of $\Delta_M$ is 
 \[
 \sum_{I \in \mathcal{I}(M)} q^{-|Y_I|} t^{n+r-|\IP(I)|} =
t^n T_M((1/q+1)t,1).
 \]
\end{corollary}
Setting $t=q$ gives the $h$-polynomial we computed earlier. Note that the bivariate information is strictly finer than the information of the normal $h$-polynomial.
\begin{proof}
First note that $\IP(I)$ is the internal passivity of its related basis. We use the well-known expansion for the Tutte polynomial,
\[
T_M(q,t) = \sum_{B \in \mathcal{B}(M)} q^{|\IA(B)|} t^{|\EA(B)|}.
\]
We have,
\begin{align*}
t^n T_M((1/q+1)t,1) =
\sum_{B \in \mathcal{B}(M)} (1/q+1)^{|\IA(B)|} t^{n+r-|\IP(B)|}\\=
\sum_{B \in \mathcal{B}(M)} \sum_{Y \subset \IA(B)}q^{-|Y|} t^{n+r-|\IP(B)|}.    
\end{align*}
Exchanging the order of summation gives the result.
\end{proof}

\section{The augmented no broken circuit complex}\label{sec:nbc}
Given a matroid $M$ on a set $E$ totally ordered by $<$, a circuit with its maximum element deleted is called a \textit{broken circuit}. The collection of subsets $S$ that contain no broken circuit form a simplicial complex $\NBC(M)$ -- the no broken circuit complex of $M$. The faces of $\NBC(M)$ are called nbc sets and it is immediate that every nbc set is independent in $M$. An independent set $I$ is at once seen to be nbc if and only if $\EA(I) = \emptyset$. The following result due to Provan is well-known; see \cite[Theorem~7.4.3]{bjorner}.
\begin{theorem}
For any matroid $M$, $\NBC(M)$ is shellable under any lexicographic ordering of its facets. Its $h$-polynomial is $T_M(q,0)$.
\end{theorem}

As a first attempt at augmenting the nbc complex of $M$, we might consider the subcomplex of $\Delta_M$ generated by those $F(I)$ where $\EA(I) = \emptyset$. This turns out to have a large number of cone points: the elements of $E(x)$ are obvious cone points, although there are more. Deleting these elements, we obtain the following definition.
\begin{definition}
 The augmented nbc complex of $M$ is the complex $\aNBC_M$ on $E(y,z)$  with facets $G(I) := y_{B_I \setminus I} z_I$ for every nbc set $I$ of $M$ related a basis $B_I$.
\end{definition}

This is a subcomplex of $\Delta_M$, pure of dimension $r-1$. The nbc complex is visibly isomorphic to the subcomplex  of $\aNBC_M$ induced by $E(z) \subset E(y,z)$. Our main results on the augmented nbc complex is the following.
\begin{theorem}
 For any linear extension $\prec$ of  $\leq_{ext/int}$ on $\NBC(M)$, the corresponding ordering of the facets of $\aNBC_M$ is a shelling.
\end{theorem}
\begin{proof}
It is sufficient to check that for any for any pair of nbc sets $I \prec K$ there is an nbc set $J \prec K$ and $c \in E$ satisfying
\[
 G(I) \cap G(K) \subset G(J) \cap G(K) = G(K)\setminus z_c.
 \]
Since $I$ and $K$ are independent sets, the proof of \cref{thm:shelling} produces $J$ and $c$ satisfying
\[
F(I) \cap F(K) \subset F(J) \cap F(K) = F(K) \setminus z_c
\]
We will show that $J$ can be chosen to be nbc, and deleting the elements of $E(x)$ from these facets gives the needed statement.
  
In the case that $J$ and $c$ are produced as in \cref{lem:case1constructJ} we have $J$ and $K$ are internally related and hence they  have equal external activity by \cref{prop:lvprop}. Since $K$ is nbc, it follows that $J$ is too. In the case that $J$ is produced as in
\cref{lem:case2constructJ} it suffices to assume that $I$ and $K$ are nbc bases and apply parts (3) and (4) of \cref{lem:acs} to see that $\EA(J) = \emptyset$ if $\EA(K) =  \emptyset$.
\end{proof}
The following two results are now immediate.
\begin{corollary}
Let $\prec$ be a linear extension of $\leq_{ext/int}$ on $\NBC(M)$. For each independent set $I$ of $M$, the restriction set of $G(I)$ in this order is equal to $z_I$.
\end{corollary}
\begin{corollary}
The $h$-polynomial of $\aNBC_M$ is equal to $T_M(1+q,0)$. That is, the $h$-vector of $\aNBC_M$ is the $f$-vector of $\NBC(M)$.
\end{corollary}
Since the restriction sets in one of our shellings of $\aNBC_M$ form a complex isomorphic to $\NBC(M)$, they are $h$-shellings as in \cref{sec:hshelling}. 

\begin{corollary} For any linear extension $\prec$ of $\leq_{ext/int}$ on $\NBC(M)$, the corresponding shelling of $\aNBC_M$ has property \eqref{eq:H'}.
\end{corollary}
\begin{proof}
    Assume $S \subset G(J)$ is codimension one. If $S$ is obtained by deleting some $y_i$ from $G(J)$ then $S \in [z_J,G(J)]$ and the implication in \eqref{eq:H'} is trivial. If $S = G(J) \setminus z_j$ for some $j \in J$ then assume $S \in [z_I,G(I)]$ for some no broken circuit set $I$. It is immediate that $J \setminus j \subset I$ by taking $z$-supports and this is what \eqref{eq:H'} demands.
\end{proof}

\section{Questions}
We close with a short selection of questions for future work.
\begin{enumerate}
 \item We can identify the lattice of flats of $M$ as a sublattice of the poset $\mathcal{I}(M)$ (with the order $\leq_{ext/int}$). Here a flat $X$ corresponds to the unique independent set $I$ satisfying $I \cup \EA(I) = X$. Is  this poset non-pure shellable? Similarly, is the external or internal order on $\mathcal{B}(M)$ (pure) shellable?
 \item Does any direct sum of tautological bundles give rise to a shellable simplicial complex, as in \cref{ssec:motivation}?
\end{enumerate}

\bibliographystyle{amsalpha}
\bibliography{biblio}
\end{document}